\documentclass{amsart}
\usepackage{amsmath, amsthm, amssymb, amsfonts, graphicx}
\usepackage[top=1in,bottom=1in,left=1in,right=1in]{geometry}
\usepackage[all]{xy}
\usepackage{tikz}
\usepackage{tikz-cd}
\usepackage{setspace}
\usepackage{hyperref}
\usepackage[foot]{amsaddr}

\def \C{\mathbb{C}}
\def\R{\mathbb{R}}

\def\Z{\mathbb{Z}}

\def\zplus{\mathbb{Z}^+}

\def\epsilon{\varepsilon}
\def\phi{\varphi}

\def\max{\text{Max}}

\def\st{\; | \;}

\def\endo{\text{End}}
\def\cfm #1{\text{CFM}(#1)}
\def\rfm #1{\text{RFM}(#1)}
\def\rcfm #1{\text{RCFM}(#1)}
\def\fsm #1{\text{FSM}(#1)}

\newtheorem{theorem}{Theorem}
\newtheorem{proposition}[theorem]{Proposition}
\newtheorem{corollary}[theorem]{Corollary}
\newtheorem{lemma}[theorem]{Lemma}

\theoremstyle{definition}
\newtheorem{example}[theorem]{Example}
\newtheorem{definition}[theorem]{Definition}
\newtheorem{remark}[theorem]{Remark}

\theoremstyle{remark}

\begin{document}

\title{On the Associativity of Infinite Matrix Multiplication}%
\author{Daniel P. Bossaller}
\email{db684513@ohio.edu}

\author{Sergio R. L\'opez-Permouth}
\email{lopez@ohio.edu}

\address{Ohio University Department of Mathematics}
\maketitle
\begin{abstract}
A natural definition of the product of infinite matrices mimics the usual formulation of multiplication of finite matrices with the caveat (in the absence of any sense of convergence) that the intersection of the support of each row of the first factor with the support of each column of the second factor must be finite. Multiplication is hence not completely defined, but restricted to a specific relation on infinite matrices. In order for the product of three infinite matrices $A$, $B$, and $C$ to behave in an associative manner, the middle factor, $B$, must link $A$ and $C$ in three ways: (i) $AB$ and $BC$ must both be defined; (ii) $A(BC)$ and $(AB)C$ must both be defined; and, finally, (iii) $A(BC)$ must equal $(AB)C$. In this article, these conditions are studied and are characterized in various ways by means of summability notions akin to those of formal calculus.
\end{abstract}

\section{Introduction.}
A central problem in linear algebra is finding the solution of a system of linear equations in several unknowns. A typical undergraduate linear algebra course covers, in detail, how to find the solution to an equation $Av = b$  for some $m \times n$ matrix $A$ and vectors $v \in F^n$ and $b \in F^m$, where $F$ is some field (usually $\R$ or $\C$). One typically finds a solution by first reducing $A$ to some ``nice" form using row operations, which can be represented by an invertible matrix $U$. This then gives a solvable system of linear equations in
\[(UA) v = Ub;\] call this solution $v = a$. However, $a$ may only be a solution to this particular matrix equation, $UA a = Ub$, which, in theory, need not be equivalent to the original system. In order to convert this new system into the original system, one must undo the row operations represented by $U$ and thus multiply by the inverse of $U$, which we call $V$. Multiplying on the left by $V$ transforms $(UA)a = Ub$ into $V(UA)a = (VU)Aa = VUb$ (note that there is an implicit assumption of associativity of multiplication in this above statement); thus $Aa = b$, and $a$ is a solution to $Av = b$.

In the context of finite matrices, this process works because multiplication of finite matrices is associative. Unfortunately the process breaks down when the matrices are infinite, indeed, consider the following matrices:
\begin{equation*}
\begin{split}
V = \begin{pmatrix}
1 &1 &1 &\cdots\\
0 &1 &1 &\cdots\\
0 &0 &1 &\cdots\\
\vdots &\vdots &\vdots &\ddots
\end{pmatrix}\text{, } &U =  \begin{pmatrix}
1 &-1 &0 &\cdots\\
0 &1 &-1 &\cdots\\
0 &0 &1 &\cdots\\
\vdots &\vdots &\vdots &\ddots 
\end{pmatrix},
\\
A = \begin{pmatrix}
0 &0 &0 &\cdots\\
-1 &0 &0 &\cdots\\
-1 &-1 &0 &\cdots\\
\vdots &\vdots &\vdots &\ddots
\end{pmatrix}, &\text{ and } b = \begin{pmatrix} 1\\ 1\\ 1\\ \vdots\end{pmatrix}.
\end{split}
\end{equation*}

One can easily see that $Av = b$ has no solution; it is also seen that $(UA)v = Ub$ has a solution, namely the zero vector $v = 0$, because $UA$ is the identity matrix and $Ub = 0$. The above matrices satisfy every step of the process outlined above; however, in the context of infinite matrices the process fails to be well-defined since $V(UA) \neq (VU)A$. The linchpin in this process seems to be the resolution of the following question: Given three matrices, when is their product associative? The proof in Section 4 from \cite{lulu} illustrates the complications that arise trying to circumvent such difficulties while solving such an infinite system, $Av = b$.

Let us introduce some notation. Denote by $M_\infty(F)$, the vector space of infinite matrices whose rows and columns are indexed by $\zplus$ with entries in the field $F$ (note that arbitrary products of these matrices need not be defined). We will denote by $\cfm F$ (respectively, $\rfm F$)  the subspace of column (row) finite matrices those matrices whose columns (rows) have finitely many nonzero elements.  We will write the intersection of $\rfm F$ and $\cfm F$, the space of row and column finite matrices, by $\rcfm F$ where each row and column has a finite number of nonzero entries; note that the matrix itself may have infinitely many nonzero entries. Also important will be the subspace of finitely supported matrices i.e., those with only finitely many nonzero entries, and we will denote these by $\fsm F$. It is of general interest to note that for a $F$-vector space $V$ of infinite dimension with basis $\{e_i \st i \in \Z^+\}$, the space of endomorphisms acting of the left of $V$, denoted $\endo_F(V)$, is isomorphic to $\cfm F$. (In the case where we have the endomorphisms acting on the right, $\endo_F(V) \simeq \rfm F$.)

In this article $F$ denotes a field of characteristic zero, and we do not endow it with a specific norm. Similar results to those in this article appear in the literature (such as in \cite{bernkopf}, \cite{cooke}), but with proofs that depend on the convergence of infinite series. 

\section{When does $AB$ Exist?}
Before discussing the associativity of matrix multiplication, we devote this section to a study of precisely when the product of two infinite matrices exists. Because we seek results in the absence of convergence, the only allowable sums will be finite sums. Take two infinite matrices $A$ and $B$, and write $A = (a_{ij})_{i,j \in \Z^+}$ and $B = (b_{jk})_{j,k \in \Z^+}$ with entries in the field $F$. The (formal) product $AB$ can be represented as the following array of formal sums.
\begin{equation}\label{matrix multiplication}
AB = \begin{pmatrix}
\sum_{j=1}^\infty a_{1j}b_{j1} &\sum_{j=1}^\infty a_{1j}b_{j2} &\sum_{j=1}^\infty a_{1j}b_{j3} &\cdots\\
\sum_{j=1}^\infty a_{2j}b_{j1} &\sum_{j=1}^\infty a_{2j}b_{j2} &\sum_{j=1}^\infty a_{2j}b_{j3} &\cdots\\
\sum_{j=1}^\infty a_{3j}b_{j1} &\sum_{j=1}^\infty a_{3j}b_{j2} &\sum_{j=1}^\infty a_{3j}b_{j3} &\cdots\\
\vdots &\vdots &\vdots &\ddots 
\end{pmatrix}.\end{equation}
Since we are working in a field, this formal product is defined in $M_\infty(F)$ only when each of the formal sums has finitely many nonzero summands. We will take this observation as our first definition for when a matrix product is ``defined."
\begin{definition}\label{defined}
If $A$ and $B$ are two matrices in $M_\infty(F)$. Then their product is \textbf{defined} if for every choice of $i$ and $k$ in $\Z^+$, $\{j \st a_{ij}b_{jk} \neq 0\}$ is finite. 
\end{definition}

An interesting phenomenon occurs when $A$ is written in terms of its columns. Say $A = (A_{*j})_{j \in \Z^+}$ where $A_{*j} = (a_{ij})_{i \in \Z^+}$ is the $j$th column of $A$. The first column of (\ref{matrix multiplication}) above may be written as
\[A_{*1}\cdot b_{11} + A_{*2} \cdot b_{21} + A_{*3} \cdot b_{31} + \cdots,\] an infinite linear combination of the columns of $A$. Given the conventions taken on $F$, one may be tempted to state that if $\{i \in \zplus \st A_{*i}b_{i1} \neq 0\}$ is an infinite set, then the product of $A$ and $B$ cannot be defined. Actually, something more subtle is required; consider the following definition from \cite{VOALepowsky}.

\begin{definition}
Let $\{V_i \st i \in I\}$ be a family of vectors indexed by the set $I$. Denote $V_i(j)$ to be the $j$th entry in the $i$th vector. Then the family $\{V_i \st i \in I\}$ is \textbf{summable} if for every $j \in \Z^+$, the set $\{i \st V_i(j) \neq 0\}$ is finite. (In the literature this set is often called the \textbf{support} of the $j$th entry in this family.) By considering entries of a matrix instead of entries of a vector, this definition can be extended easily to the context of infinite matrices.
\end{definition}

\begin{remark}
In this article, for notational convenience and to preserve the intuition that formed the following results, we will often write an entry of a vector or matrix as a formal sum as in equation (\ref{matrix multiplication}), regardless of whether the entry is finitely or infinitely supported.
\end{remark}

Using the definition above one may give equivalent characterizations of the families of row and column finite matrices.

\begin{lemma}
A matrix is column finite if and only if its family of rows is summable. Likewise, a matrix is row finite if and only if its family of columns is summable.
\end{lemma}

With this new definition of a summable family of vectors in hand, we now have the tools to further analyze conditions for matrix multiplication to exist.

\begin{definition}\label{leftrightdefined}
Let $A, B \in M_\infty(F)$. Then the product $AB$ is \textbf{right defined} if for every $k \in \Z^+$, the family $\{A_{*j}b_{jk} \st j \in \zplus\}$ is summable. Similarly, we say $AB$ is \textbf{left defined} if for each $i \in \zplus$, the family $\{a_{ij}B_{j*} \st j \in \Z^+\}$ is summable. 
\end{definition}

As expected from the discussion of these three definitions of the product of two matrices $A$ and $B$, we have the following proposition, which will grant us some flexibility for the remainder of the article. 
\begin{proposition} \label{defined leftrightdefined}
Let $A = (a_{ij})_{i,j \in \zplus}$ and $B = (b_{jk})_{j,k \in \zplus}$. Then the following are equivalent:
\begin{enumerate}
\item $AB$ is defined
\item $AB$ is right defined
\item $AB$ is left defined.
\end{enumerate}
\end{proposition}
\begin{proof}
This proof will show $(1) \Leftrightarrow (2)$, and a symmetric argument gives $(1) \Leftrightarrow (3)$. Say that $AB$ is defined; let $k \in \zplus$, and consider the family $\{A_{*j}b_{jk} \st j \in \zplus\}$. In order to show this family is summable, we must show that the support of the $i$th entry is finite, in other words we must assure that for some fixed $i$ and $k$, $\{j \st a_{ij}b_{jk} \neq 0, j \in \zplus\}$ is finite; thus, by assumption, we have our desired conclusion. Through a similar argument we have the reverse implication. 
\end{proof}

With this proposition in hand, consider, for instance, the following simpler proof of a well-known result.

\begin{corollary}\label{rstar starc}
Let $A \in \rfm F$ and $B \in \cfm F$. Then for any matrix $M \in M_\infty(F)$, $AM$ and $MB$ are both defined.
\end{corollary}
\begin{proof}
Consider the product $AM$. Write $M = (M_{j*})_{j \in \zplus}$ and take $i \in \zplus$ arbitrary but fixed. Note that $\{j \st a_{ij} \neq 0\}$ is a finite set since $A$ is row finite. But this means that $\{j \st a_{ij}M_{j*} \neq 0\}$ is finite and thus the family is summable. So $AM$ is defined. The argument for $MB$ is similar. 
\end{proof}

\section{Associativity.}
Take three arbitrary matrices $A, B, C \in M_\infty(F)$; we saw in the previous section that there is no guarantee that $AB$ and $BC$ need be defined. Even if they were defined, we saw three matrices, $V, U,$ and $A$ in our very first example where $VU$, $UA$, $V(UA)$, and $(VU)A$ are defined but $V(UA) \neq (VU)A$. It is therefore necessary to coin some terminology.
\begin{definition} A triple $(A,B,C)$ of invertible matrices is an \textbf{associative triple} if it satisfies the following three properties.
\begin{enumerate}
\item $AB$ and $BC$ are defined,
\item $A(BC)$ and $(AB)C$ are defined, and
\item $A(BC) = (AB)C$. 
\end{enumerate}
When the triple $(A,B,C)$ is an associative triple, we simply write $A(BC) = (AB)C$ as the first two conditions are implicit in the third.
\end{definition}
\begin{remark}
This definition of an associative triple might feel a bit pedantic; after all, the statement ``$A(BC) = (AB)C$" implicitly assumes that the products $AB$, $BC$, $A(BC)$, and $(AB)C$ are all defined. The reason for the stratification of the definition is that over the next few sections we will study three successively larger vector spaces based on these conditions.
\end{remark}

Consider the following example from \cite{keremedis}.
\begin{example}\label{multiplication not associative}
Let
\[A = \begin{pmatrix}
1 &1 &1 &\cdots\\
0 &0 &0 &\cdots\\
0 &0 &0 &\cdots\\
\vdots &\vdots &\vdots &\ddots
\end{pmatrix}\text{, } B = \begin{pmatrix}
0 &1 &0 &\cdots\\
-1 &0 &1 &\cdots\\
0 &-1 &0 &\cdots\\
\vdots &\vdots &\vdots &\ddots
\end{pmatrix} \text{, and } C = \begin{pmatrix}
1 &0 &0 &\cdots\\
1 &0 &0 &\cdots\\
1 &0 &0 &\cdots\\
\vdots &\vdots &\vdots &\ddots
\end{pmatrix}.\]

Note that $(A,C,B)$ is not an associative triple, since $AC$ is not defined. However, $CB$ is defined and row and column finite; thus $A(CB)$ is defined (by Corollary \ref{rstar starc}). The triple $(A,B,C)$ also provides a good example of a triple of three matrices that satisfy the first two properties but fail the third. Because $B$ is both row and column finite, $AB$ and $BC$ are defined, satisfying the first property. Because both $AB$ and $BC$ are row and column finite, $A(BC)$ and $(AB)C$ are defined, but explicit calculation shows that these two products are not the same. 
\end{example}
This section and the next will illustrate that there are two different but equally clarifying lenses through which one can examine the triple $(A,B,C)$ of infinite matrices. The focus of this section will be the properties of $A$ and $C$ that guarantee associativity. The focus of the next section will be the role that the middle matrix, $B$, plays in determining the associativity of the triple. With this in mind, we introduce the following definitions.
\begin{definition}
Let $A$ and $C$ be fixed but arbitrary elements of $M_\infty(F)$. We say $B$ is a \textbf{link} between $A$ and $C$ if $AB$ and $BC$ are both defined. We say that $B$ is a \textbf{strong link} between $A$ and $C$ if $B$ is a link and $A(BC)$ and $(AB)C$ are both defined. We call $B$ an \textbf{associative link} between $A$ and $C$ if $B$ is a strong link and $A(BC) = (AB)C$. We will denote the family of links by $G_2(A,C)$, the family of strong links by $G_4(A,C)$ and the family of associative links by $G_5(A,C)$. The subscripts $2, 4,$ and $5$ refer to the number of conditions required to belong each family.
\end{definition}

It is a quick check to see that $G_i(A,C)$ is a vector subspace of $M_\infty(F)$ for each $i \in \{2, 4, 5\}$. Note that scalar multiplication is equivalent to multiplication by $\text{Diag}(\lambda, \lambda, \lambda, \ldots)$ where $\lambda \in F$.

\begin{example}
By construction of our families, one can see that $G_5(A,C) \subseteq G_4(A,C) \subseteq G_2(A,C)$; moreover, there exist matrices $A$ and $C$ for which the containments are proper. Let the matrices $A$ and $C$ be as in the previous example. Then note that
\[B = \begin{pmatrix}
0 &1 &0 &\cdots\\
-1 &0 &1 &\cdots\\
0 &-1 &0 &\cdots\\
\vdots &\vdots &\vdots &\ddots
\end{pmatrix}\text{ and }
B' = \begin{pmatrix}
0 &1 &0 &\cdots\\
1 &0 &1 &\cdots\\
0 &1 &0 &\cdots\\
\vdots &\vdots &\vdots &\ddots
\end{pmatrix}\] are examples of matrices contained in $G_4(A,C) \backslash G_5(A,C)$ and $G_2(A,C) \backslash G_4(A,C)$, respectively.
\end{example}

If $A$ and $C$ are arbitrary elements of $M_\infty(F)$, we have $\rcfm F \subseteq G_2(A,C)$ by Corollary \ref{rstar starc}. If there is any matrix $B$ that fails to be row finite or column finite, one may find $A$ or $C$ such that $AB$ or $BC$ is not defined. Indeed, if some row (or column) of $B$ has non-finite support, one need merely include a column of all 1's in $C$ (or row of all 1's in $A$) to make sure that $AB$ or $BC$ is not defined. These two observations together give
\[\bigcap_{A,C \in M_\infty(F)} G_2(A,C) = \rcfm F.\]
A similar phenomenon occurs with $G_4(A,C)$.
\begin{proposition}\label{cap G4 fsm}
\[\bigcap_{A,C \in M_\infty(F)} G_4(A,C) = \fsm F.\]
\end{proposition}
\begin{proof}
One inclusion is due to Corollary \ref{rstar starc} because $B \in \fsm F$ implies that $AB$ and $BC$ are defined; moreover, as $AB \in \rfm F$ and $BC \in \cfm F$, $\fsm F \subseteq G_4(A,C)$ for any choice of $A$ and $C$. On the other hand, say that $B$ fails to have finite support. In light of the work above, one need only consider the case where $B \in \rcfm F$. Consider the two sets $J = \{j \st B_{*j} \neq 0\}$ and $K = \{k \st B_{*k} \neq 0 \}$. Since it was assumed that $B$ has infinite support, at least one of the sets $J$ or $K$ must be infinite. Say $|J| = \infty$. (The proof of the case where $|K| = \infty$ will be symmetric.) 

Let $\{b_{jk_j}^1 \st j \in J\}$ be the first nonzero entry in each of the rows $B_{j*}$ for $j \in J$. Then define 
\[C = \sum_{j \in J} E_{{k_j}1}\] where $E_{ab}$ is the matrix unit with $1$ in the $(a,b)$ position and zeroes elsewhere. Since $B \in \rcfm F$, $BC$ is defined; moreover, the first column of $BC$ has infinitely many nonzero entries due to the construction of $C$. One need only take $A$ as in the previous two examples to see that $A(BC)$ is not defined.
\end{proof}

The next two results are folklore, but we will include proofs for completeness. Before those results we make the following definition.

\begin{definition}
Let $A$ be a row finite matrix, and let $A_{i*}$ be the $i$th row of $A$. The {\bf length} of $A_{i*}$ is the largest number $n_i \geq 0$ such that for every $j > n_i$, we have that $a_{ij} = 0$. A symmetric definition gives the length of the $j$th column of a column finite matrix $B$.
\end{definition}

\begin{proposition}\label{rstarc associative}
Let $A \in \rfm F$ and $C \in \cfm F$. Then $G_5(A,C) = M_\infty(F)$. 
\end{proposition}
\begin{proof}
Note that, by Corollary \ref{rstar starc}, $(AB)C$ and $A(BC)$ are both defined for all $B \in M_\infty$. So consider the $(i,l)$th entry of both of those terms. Let $n_i$ be the length of the $i$th row of $A$ and $m_l$ be the length of the $l$th column of $C$. Then the $(i,l)$ entry of $A(BC)$ can be written as
\begin{equation*}
\begin{split} 
\left(A(BC)\right)_{il} = \sum_{j=1}^\infty a_{ij}\left(\sum_{k=1}^\infty b_{jk}c_{kl}\right) 
&= \sum_{j=1}^{n_i}  a_{ij}\left(\sum_{k=1}^{m_l} b_{jk}c_{kl}\right)\\
= \sum_{k=1}^{m_l} \left(\sum_{j=1}^{n_i} a_{ij}b_{jk}\right)c_{kl}
&= \left((AB)C\right)_{il}.
\end{split}
\end{equation*}
\end{proof}

\begin{corollary}\label{fsm associative}
If $B \in \fsm F$, then for any matrices $A$ and $C$, $A(BC) = (AB)C$. In particular,
\[\bigcap_{A,C \in M_\infty(F)} G_5(A,C) = \fsm F.\]
\end{corollary}
\begin{proof}
To prove the first statement, note that $B \in \fsm F$; this means that one may write \[B = \left(\begin{array}{c|c}
B' &0\\ \hline
0 &0
\end{array}\right)\] where $B'$ is an $n \times n$ matrix. Divide the matrices $A$ and $C$ into block matrices similarly:
\[A = \left(\begin{array}{c|c}
A_1 &A_2\\ \hline
A_3 &A_4
\end{array}\right) \text{ and } C = \left(\begin{array}{c|c}
C_1 &C_2\\ \hline
C_3 &C_4
\end{array}\right).\]
So \[A(BC) = \left(\begin{array}{c|c}
A_1(B'C_1) &A_1(B'C_2)\\ \hline
A_3(B'C_1) &A_3(B'C_2)
\end{array}\right).\] As $A_1, A_3 \in \rfm F$ and $C_1,C_2 \in \cfm F$, Proposition \ref{rstarc associative} gives $A(BC) = (AB)C$.

Now, since $\fsm F \subseteq G_5(A,C) \subseteq G_4(A,C)$ for arbitrary matrices $A, C \in M_\infty(F)$, Proposition \ref{cap G4 fsm} immediately gives the second statement.
\end{proof}

Proposition \ref{rstarc associative} and Corollary \ref{fsm associative} are dual to each other in the sense that the first relies on finiteness in the outer matrices of the triple $(A,B,C)$ to achieve associativity while the second relies on finiteness in the center matrix $B$ to achieve associativity. In the next section, we will introduce machinery that will lead to an even more straightforward proof of Corollary \ref{fsm associative}

\section{Summability and Associativity.}
In the definitions of $G_i(A,C)$ for $i \in \{2,4,5\}$ in the previous section, our focus was on the matrices $A$ and $C$ and their influence on the associativity of infinite matrices. This section will explore the role of $B$ in the presence (or absence) of the associativity condition.

Recall Definition \ref{leftrightdefined} and Proposition \ref{defined leftrightdefined}. For fixed but arbitrary $A,C \in M_\infty(F)$, one can see that $B$ is a link if and only if for every $k \in \zplus$ the family $\{A_{*j}b_{jk} \st j \in \zplus\}$ is summable and for every $j \in \zplus$, $\{b_{jk}C_{k*} \st k \in \zplus\}$ is summable. With this and the goal of an associativity condition in mind, one may be tempted to conjecture that $A(BC) = (AB)C$ if and only if the family $\{A_{*j}b_{jk}C_{k*} \st j,k \in \zplus\}$ is summable. The following example shows that this condition fails to be sufficient for associativity.

\begin{example} \label{(D) not sufficient}
Let \[A = \begin{pmatrix}
0 &1 &1 &\cdots\\
0 &0 &0 &\cdots\\
0 &0 &0 &\cdots\\
\vdots &\vdots &\vdots &\ddots
\end{pmatrix}\text{, } B= \begin{pmatrix}
1 &1 &1 &\cdots\\
1 &1 &0 &\cdots\\
1 &0 &0 &\cdots\\
\vdots &\vdots &\vdots &\ddots 
\end{pmatrix} \text{, and } C = \begin{pmatrix}
0 &0 &0 &\cdots\\
1 &0 &0 &\cdots\\
1 &0 &0 &\cdots\\
\vdots &\vdots &\vdots &\ddots 
\end{pmatrix}.\]

Clearly $AB$ and $BC$ do not exist, so $(A,B,C)$ fails to be an associative triple. What is interesting is that the family $\{A_{*j}b_{jk}C_{k*} \st j,k \in \zplus\}$ is summable. When considering this family of matrices we may think of the entry $b_{jk}$ of $B$ as a ``connector" of sorts which multiplies the $j$th column of $A$ by the $k$th row of $C$ and then multiplies the resulting matrix by some scalar from the field. In order to check the summability of this family, we only need to multiply the columns and rows indicated by the nonzero entries of $B$. These nonzero entries fall into three (not disjoint) families: the first row of $B$, $\{b_{1k} \st k \in \zplus\}$; the first column of $B$, $\{b_{j1} \st j \in \zplus\}$; and the singleton, $\{b_{22}\}$.

For $b_{1k}$ taken from the first row of $B$, the family $\{A_{*1}b_{1k}C_{k*} \st k \in \zplus\} = \{0\}$ since the first column of $A$ is zero. Because the first row of $C$ is also zero, the family generated by any entry taken from the first column of $B$ will be the set containing the singleton zero. Consider $b_{22}$; then 

\[A_{*2}b_{22}C_{2*} = \begin{pmatrix} 1\\ 0\\ 0\\ \vdots \end{pmatrix} \cdot \begin{pmatrix} 1 &0 &0 &\cdots \end{pmatrix} = \begin{pmatrix}1 &0 &0 &\cdots\\ 0 &0 &0 &\cdots\\ 0 &0 &0 &\cdots\\ \vdots &\vdots &\vdots &\ddots \end{pmatrix},\] which is consequently the only nonzero element of the family $\{A_{*j}b_{jk}C_{k*} \st j,k \in \zplus\}$, so the family must be summable.
\end{example}

On its own, the requirement that $\{A_{*j}b_{jk}C_{k*} \st j,k \in \zplus\}$ be a summable family of matrices seems to have nothing to do with the existence of the associative triple $(A,B,C)$; it cannot even assure that $AB$ and $BC$ are defined. (From here on we will refer to this summability requirement as \textbf{condition $\mathbf{(D)}$}.) This condition is nonetheless intertwined with associativity, as the following proposition shows.

\begin{proposition}\label{associative implies AB BC (D)}
If $A(BC) = (AB)C$, then $\{A_{*j}b_{jk}C_{k*} \st j,k \in \zplus\}$ is a summable family of matrices.
\end{proposition}
\begin{proof}
Say, in anticipation of a contradiction, $A(BC) = (AB)C$, but that condition (D) fails to be satisfied. Then, in particular, $BC$ is defined, and from there we will then find a contradiction in the definedness of $A(BC)$.

With that in mind, say that condition (D) fails for the $(i,l)$ entry of the family of matrices. Consider the formal expression of that entry of $A(BC)$:
\[(A(BC))_{il} = \sum_{j=1}^\infty a_{ij}\left(\sum_{k=1}^\infty b_{jk}c_{kl}\right).\] Since $BC$ is defined, then for every $j \in \zplus$ there exists a smallest number $m_j$ such that 
\[\sum_{j=1}^\infty a_{ij}\left(\sum_{k=1}^\infty b_{jk}c_{kl}\right) = \sum_{j=1}^\infty a_{ij}\left(\sum_{k=1}^{m_j} b_{jk}c_{kl}\right) = \sum_{j=1}^\infty \sum_{k=1}^{m_j} a_{ij}b_{jk}c_{kl}.\]
Because condition (D) fails at the $(i,l)$ entry, there cannot exist any number $n \in \zplus$ such that 
\[\sum_{j=1}^\infty \sum_{k=1}^{m_j} a_{ij}b_{jk}c_{kl} = \sum_{j=1}^n \sum_{k=1}^{m_j} a_{ij}b_{jk}c_{kl},\] which in turn means that $A(BC)$ is not defined.  This is a contradiction, so associativity implies condition (D).  
\end{proof}

The next proposition will show that, with an additional assumption, the converse to the above proposition also holds.

\begin{proposition} \label{(D) AB BC implies G5}
If $AB$ and $BC$ are both defined and $\{A_{*j}b_{jk}C_{k*} \st j,k \in \zplus\}$ is summable, then $A(BC) = (AB)C$. 
\end{proposition}
\begin{proof}
Consider the $(i,l)$ entry of $A(BC)$ and $(AB)C$. If for each choice of $i,l \in \zplus$, they are equal, then we have associativity.

Consider the formal expression of the $(i,l)$ entry of $A(BC)$. Then since $BC$ is defined, for each $j \in \zplus$ there exists a smallest number $m_j$ such that the following equality holds.
\[(A(BC))_{il} = \sum_{j=1}^\infty a_{ij}\left(\sum_{k=1}^\infty b_{jk}c_{kl}\right) = \sum_{j=1}^\infty a_{ij}\left(\sum_{k=1}^{m_j} b_{jk}c_{kl}\right) = \sum_{j=1}^\infty \sum_{k=1}^{m_j} a_{ij}b_{jk}c_{kl}.\] Now, due to the summability of $\{A_{*j}b_{jk}C_{k*} \st j,k \in \zplus\}$, each entry has only finite support, so there must exist a smallest number $n$ such that 
\[\sum_{j=1}^\infty \sum_{k=1}^{m_j} a_{ij}b_{jk}c_{kl} = \sum_{j=1}^n \sum_{k=1}^{m_j} a_{ij}b_{jk}c_{kl}.\] From here one may define $m = \max\{m_j \st 1 \leq j \leq n\}$, giving the following equality:
\[\sum_{j=1}^n \sum_{k=1}^{m_j} a_{ij}b_{jk}c_{kl} = \sum_{j=1}^n \sum_{k=1}^m a_{ij}b_{jk}c_{kl}.\]
Through a similar process, one may write the $(i,l)$ entry of $(AB)C$ as 
\begin{equation*}
\begin{split}
((AB)C)_{il} &= \sum_{k=1}^\infty \left(\sum_{j=1}^\infty a_{ij}b_{jk}\right)c_{kl} 
= \sum_{k=1}^{m'}\sum_{j=1}^{n_k'} a_{ij}b_{jk}c_{kl}\\ 
&= \sum_{k=1}^{m'}\sum_{j=1}^{n'} a_{ij}b_{jk}c_{kl}
= \sum_{j=1}^{n'} \sum_{k=1}^{m'} a_{ij}b_{jk}c_{kl}
\end{split}
\end{equation*} 
where the choice of the $n_k'$ is due to the definedness of $AB$, the choice of $m'$ is due to condition $(D)$, and $n' = \max\{n_k \st 1 \leq k \leq m'\}$. These constants, $n$, $m$, $n'$, and $m'$ which were defined in the previous argument depend only on $i$ and $l$, the two indices which determine the entry of the product. If we can show that $n = n'$ and $m = m'$, then equality follows immediately by the commutativity of finite addition.

We claim that $n = n'$ and $m = m'$, giving
\begin{equation}\label{associative equality}\sum_{j=1}^n \sum_{k=1}^m a_{ij}b_{jk}c_{kl} = \sum_{j=1}^{n'} \sum_{k=1}^{m'} a_{ij}b_{jk}c_{kl}.\end{equation} Say, in anticipation of a contradiction, that equation (\ref{associative equality}) does not hold, so one of following eight cases must occur.\[\begin{array}{c c c c c}
1.\ n' > n \text{ and } m'= m & & & &5.\ n> n'\text{ and }  m = m'\\
2.\ n' = n\text{ and } m' > m & & & &6.\ n=n' \text{ and } m > m'\\
3.\ n' > n\text{ and } m' > m & & & &7.\ n>n' \text{ and } m> m'\\
4.\ n > n'\text{ and } m' > m & & & &8.\ n'>n \text{ and } m>m'\\
\end{array}\] We will check cases (1), (2), and (4). The proof for (3) is similar to proofs for (1) and (2), and the proofs (5) through (8) are symmetric to the proofs of the first four cases.

In the first case, note that 
\[\sum_{j=1}^n \sum_{k=1}^m a_{ij}b_{jk}c_{kl} + \sum_{j=n+1}^{n'} \sum_{k=1}^m a_{ij}b_{jk}c_{kl} = \sum_{j=1}^{n'} \sum_{k=1}^{m'}a_{ij}b_{jk}c_{kl},\] but then note that by construction $a_{ij}b_{jk} = 0$ for all $j > n$. Thus equation (\ref{associative equality}) holds since the second summand is zero. Something similar happens in the second case, where it is assumed that $n' = n$ and $m' > m$. Write
\[\sum_{j=1}^{n'} \sum_{k=1}^{m'} a_{ij}b_{jk}c_{kl} + \sum_{j=1}^{n'} \sum_{k=m' + 1}^{m}a_{ij}b_{jk}c_{kl} = \sum_{j=1}^n \sum_{k=1}^m a_{ij}b_{jk}c_{kl},\] and note that, for all $k> m'$, we have $b_{jk}c_{kl} = 0$, giving equation (\ref{associative equality}) again. The third case is clear from the two previous cases.

Now consider the fourth case, where $n > n'$ and $m' > m$. Consider the following two expansions
\begin{equation}\label{interleave compare}
\begin{split}
\sum_{j=1}^{n'} \sum_{k=1}^{m'} a_{ij}b_{jk}c_{kl} + \sum_{j=n'+1}^n \sum_{k=1}^{m'} a_{ij}b_{jk}c_{kl} &= \sum_{j=1}^n \sum_{k=1}^{m'}a_{ij}b_{jk}c_{kl}\\
\sum_{j=1}^{n} \sum_{k=1}^{m} a_{ij}b_{jk}c_{kl} + \sum_{j=1}^n \sum_{k=m+1}^{m'} a_{ij}b_{jk}c_{kl} &= \sum_{j=1}^n \sum_{k=1}^{m'}a_{ij}b_{jk}c_{kl}
\end{split}
\end{equation}
First note that the two equations above are equal to each other. Additionally, for each choice of $j>n'$, it follows that $a_{ij}b_{jk} = 0$, and for each $k > m$, we have $b_{jk}c_{kl} = 0$ by the construction above. Thus the second summands in both equations of (\ref{interleave compare}) are both zero, giving the desired equality.
\end{proof}

\begin{remark}
It is interesting to see how condition (D) actually works in the above context. The example at the beginning of this section shows that, on its own, condition (D) cannot assure that multiplication is even defined. Condition (D) works best in conjunction with other assumptions. For example, let 
\[\sum_{j=1}^\infty a_{ij}\left(\sum_{k=1}^\infty b_{jk}c_{kl}\right)\] be the formal sum that accounts for all of the elements contributing to the $(i,l)$ entry of $A(BC)$. In both of the arguments we were able to assure that in the double sum at least one of the sums is finite. For instance, in the above example $BC$ being defined would imply that for each $j \in \zplus$ there exists $m_j$ to turn the double infinite sums into an infinite collection of finite sums. Then condition (D) assures that the collection of elements $\{a_{ij}b_{jk}c_{kl} \st j,k \in \zplus\}$ must have only finitely many nonzero elements, limiting the second sum to a finite number of entries.
\end{remark}
The previous two results combine to give sufficient and necessary conditions for associativity.
\begin{theorem}
Let $A$, $B$, and $C \in M_\infty(F)$. Then $A(BC) = (AB)C$ if and only if $AB$ and $BC$ are defined and the family $\{A_{*j}b_{jk}C_{k*} \st j,k \in \zplus\}$ is summable.
\end{theorem}

The first consequence of the theorem gives an alternative proof of Corollary \ref{fsm associative}.

\begin{corollary}
Let $B \in \fsm F$. Then for any matrices $A$ and $C$ in $M_\infty(F)$, $A(BC) = (AB)C$.
\end{corollary}
\begin{proof}
Since $B \in \fsm F$, $AB$ and $BC$ are defined; moreover, condition (D) holds since $B$ has finitely many nonzero entries. Thus $(A,B,C)$ forms an associative triple of matrices by the theorem.
\end{proof}

The next two results show that given three matrices $A$, $B$, and $C$, so long as two of matrices are sufficiently well-behaved, associativity follows. The second is especially interesting, since it shows how the theorem gives insight into the problem mentioned in the Introduction.

\begin{corollary}
Let $A,B \in \rfm F$. Then for any $C \in M_\infty(F)$, $A(BC) = (AB)C$. Let $B,C \in \cfm F$. Then for any $A \in M_\infty(F)$, $A(BC) = (AB)C$.
\end{corollary}
\begin{proof}
In light of the above theorem, one need only show that condition (D) holds. Say that the length of the $i$th row is $n$ (we may do this as $A$ is row finite); then for every $j > n$, $a_{ij}b_jk = 0$. Note, however, that for every $j \leq n$ the rows $B_{j*}$ have finite support by assumption. Thus for $j\leq n$ there are finitely many $b_{jk} \neq 0$. So this means that the entry has finite support and the family is summable. The proof of the second claim follows from a symmetric argument.
\end{proof}

\begin{proposition}
If $BC$ is defined and $A$ is row finite, then $A(BC) = (AB)C$.
\end{proposition}
\begin{proof}
Note that $A$ is row finite, so $AB$ is defined. In light of the theorem, all that remains to check is that condition (D) is satisfied. Consider the support of the $(i,l)$ entry, and let $I = \{j,k \in \zplus \st a_{ij}b_{jk}c_{kl} \neq 0\}$. Since $BC$ is defined, the set $\{k \st b_{jk}c_{kl} \neq 0\}$ is finite for every choice of $j \in \zplus$. Since $A$ is row finite, $\{j \st a_{ij} \neq 0\}$ is finite, thus the set $I$ is finite and our family of matrices is summable.
\end{proof}

To conclude, let us apply the proposition to the problem mentioned in the introduction. Recall that given any system of linear equations
\[Av = b,\] we wish to transform the matrix $A$, through a matrix of row operations $U$, into a form which easily begets a solution. Say $(UA)v = Ub$ has a solution in the vector $\bar a$. Given an inverse matrix $V$ for $U$, there is no guarantee that the solution $\bar a$ will be a solution to the original system of linear equations since $(V,U,A)$ need not be an associative triple of matrices. However, if $UA$ is defined and $U$ has a row finite inverse $V$, then if $\bar a$ is a solution to $(UA)v = Ub$, $\bar a$ is a solution to $(VU)Av = (VU)b$. Thus $\bar a$ is a solution to $Av = b$.

We end this article with an idea for further work on this problem. The reader will note that the introduction included two assumptions. First we assumed that we were working over a field $F$, and moreover, that the characteristic of $F$ is zero. It would be interesting to explore how the theory withstands the weakening of assumptions on $F$. What if $F$ were a field of prime characteristic? How much of the theory could one recover by assuming that $F$ is an arbitrary ring?


\begin{thebibliography}{1}

\bibitem{lulu}
Al-Essa, L.~M., L\'opez-Permouth, S.~R., Muthana, N.~M. (2018).
\newblock Modules over infinite-dimensional algebras.
\newblock {\em Linear Multilinear Algebra}, 66(3): 488-496.

\bibitem{bernkopf}
Bernkopf, M. (1968).
\newblock A history of infinite matrices.
\newblock {\em Arch. Hist. Exact Sci.} 4(4): 308--358.

\bibitem{cooke}
Cooke, R. (1955).
\newblock {\em Infinite Matrices and Sequence Spaces}.
\newblock Mineola, NY: Dover.

\bibitem{keremedis}
Keremedis, K., Abian, A. (1988).
\newblock On the associativity and commutativity of multiplication of infinite
  matrices.
\newblock {\em Internat. J. Math. Ed. Sci. Tech.} 19: 175--197.

\bibitem{VOALepowsky}
Lepowsky, J., Li, H. (2012).
\newblock {\em Introduction to Vertex Operator Algebras and Their
  Representations}, Vol. 227.
\newblock New York: Springer Science \& Business Media.

\end{thebibliography}
\end{document}